\documentclass{amsart}

\usepackage[T1]{fontenc}
\usepackage[utf8]{inputenc}

\usepackage[english]{babel}

\usepackage[leqno]{amsmath}
\usepackage{amsthm}
\usepackage{amsfonts}
\usepackage{amssymb}

\usepackage{color}
\usepackage{wrapfig}

\usepackage{graphicx}
\DeclareGraphicsExtensions{.eps}

\usepackage{subfig}
\usepackage{caption}

\usepackage{hyperref}
\hypersetup{
	unicode=true,
	colorlinks=true,
	urlcolor=blue,
	pdfauthor=Ewa Nowakowska,
	pdftitle=Dimension reduction for unknown clusters,
	linkcolor=blue
}

\theoremstyle{plain}
\newtheorem{lemma}{Lemma}[section]

\newtheorem{proposition}[lemma]{Proposition}
\newtheorem{corollary}[lemma]{Corollary}

\theoremstyle{remark}
\newtheorem{remark}[lemma]{Remark}

\theoremstyle{definition}

\providecommand{\norm}[1]{\left\lVert#1\right\rVert}
\providecommand{\abs}[1]{\left\lvert#1\right\rvert}
\DeclareMathOperator{\tr}{tr}
\DeclareMathOperator{\var}{var}

\DeclareMathOperator{\diag}{diag}
\DeclareMathOperator{\rank}{rank}
\DeclareMathOperator{\sd}{sd}
\DeclareMathOperator{\sdist}{sdist_o}
\DeclareMathOperator{\sss}{sss}

\title[Dimension reduction for unknown clusters]{Dimension reduction for data of unknown cluster structure}
\begin{document}

\author{Ewa Nowakowska}
\author{Jacek Koronacki}
\author{Stan Lipovetsky}

\address{Ewa Nowakowska\\Institute of Computer Science\\Polish Academy of Sciences\\
		ul. Jana Kazimierza 5\\01-248 Warszawa\\Poland}
\email{ewa.nowakowska@ipipan.waw.pl}

\address{Jacek Koronacki\\Institute of Computer Science\\Polish Academy of Sciences\\
		ul. Jana Kazimierza 5\\01-248 Warszawa\\Poland}	
\email{jacek.koronacki@ipipan.waw.pl}

\address{Stan Lipovetsky\\GfK Custom Research North America\\Marketing \& Data Sciences\\8401 Golden Valley Rd.\\Minneapolis MN 55427\\USA}	
\email{stan.lipovetsky@gfk.com}

\subjclass[2000]{62H25, 62H30}

\begin{abstract}
 	For numerous reasons there raises a need for dimension reduction that preserves certain characteristics of data. In this work we focus on data coming from a mixture of Gaussian distributions and we propose a method that preserves distinctness of clustering structure, although the structure is assumed to be yet unknown. The rationale behind the method is the following: (i) had one known the clusters (classes) within the data, one could facilitate further analysis and reduce space dimension by projecting the data to the Fisher's linear subspace, which --- by definition --- preserves the structure of the given classes best (ii) under some reasonable assumptions, this can be done, albeit approximately, without the prior knowledge of the clusters (classes). In the paper, we show how this approach works. We present a method of preliminary data transformation that brings the directions of largest overall variability close to the directions of the best between-class separation. Hence, for the transformed data, simple PCA provides an approximation to the Fisher’s subspace.  We show that the transformation preserves distinctness of unknown structure in the data to a great extent.
\end{abstract}

\keywords{dimension reduction, Gaussian mixture models, Fisher's subspace, principal component analysis}

\maketitle

%%%%%%%%%%%%%%%%%%%%%%%% INTRODUCTION %%%%%%%%%%%%%%%%%%%%%%%%
\section{Introduction}

%%%%%%%% state of the art
\subsection{State-of-the-art}

Dimension reduction techniques, also referred to as feature extraction algorithms, are a common way of reducing intrinsic complexity of data and consequently facilitating its further analysis. It is typically expected that certain characteristics of data will be preserved in the process. In particular, for data exhibiting clustering structure, the structure is expected to be preserved to a largest possible extent. Frequently it is captured in terms of distances between observations as in \cite{bryant}, which describes one of first methods for linear feature extraction in this context. Another line of works starts with \cite{odell79} that proposes a transformation for continuous data that lowers the dimension without increasing the probabilities of misclassification. The approach is further developed in \cite{decell81}, \cite{young} and \cite{tubbs}. Among more recent works \cite{letters} proposes a method of dimension reduction that preserves clustering structure, however it takes the common assumption of known cluster assignments. Finally \cite{wang} presents an interesting overview of methods in an application to a pattern recognition task. 

The attempt to approach the problem of dimension reduction trying to preserve distinctness of the structure originates in a series of works on learning mixture parameters in an appropriate subspace. In \cite{kalai} one-dimensional random projections were considered and then in \cite{moitra} generalized to arbitrary number of clusters. Based on Johnson-Lindenstrauss (concentration) theorem, \cite{dasgupta} suggested random projections to substantially lower -- but in general --- more than one-dimensional subspace. In \cite{kannan} the distributional assumptions were relaxed, however the main assumption of high initial cluster separation intrinsic for concentration theorem remained. Only in \cite{brand} random projections were replaced with spectral approach, making substantial progress in relaxing the requirement of initial cluster separation. It was first applied in \cite{vempala} and then the results were improved in \cite{ach} and \cite{kannan1}. A breakthrough was made by \cite{vemp}. The authors presented an affine invariant parameter learning algorithm where the preliminary data transformation was used to enhance the distinctness of the clustering structure and thereby further relaxing the separability assumptions. From our perspective it meant that it is possible to sharpen the clustering structure without actually knowing it. This significant discovery has become the major inspiration for the method proposed in the next sections.

%%%%%%%% model
\subsection{Model and notation}

We consider a data set $X = ( x_1, \ldots, x_n)^T, \ X \in \mathbb{R}^{n \times d}$ of $n$ observations coming from a mixture of $k$ $d$-dimensional normal distributions
\begin{equation*}
	f(x) = \pi_1 f_1(\mu_1, \mathbf{\Sigma}_1)(x) + \ldots + \pi_k f_k(\mu_k, \mathbf{\Sigma}_k)(x),
\end{equation*}
where
\begin{equation*}
	f_l(\mu_l,\mathbf{\Sigma}_l)(x) = \frac{1}{(\sqrt{2 \pi})^d \sqrt{\det\mathbf{\Sigma}_l}} e^{-\frac{1}{2}(x-\mu_l)^T
  \mathbf{\Sigma}_l^{-1}(x-\mu_l)}.
\end{equation*}
We call each $f_l(\mu_l, \mathbf{\Sigma}_l)$, $l = 1, \ldots, k$ a component of the mixture and each $\pi_l$, $l = 1, \ldots, k$ a mixing factor of the corresponding component (see \cite{kmb} or \cite{htf} and \cite{Lipovetsky:2013aa} or \cite{Lipovetsky:2012aa} for alternatives). We assume that for all the components equal mixing factors are assigned $\pi_1 = \dots = \pi_k = \frac{1}{k}$. However, we allow different covariance matrices $\mathbf{\Sigma}_l$. Additionally we assume large space dimension with respect to the number of components $d > k-1$ to leave room for dimension reduction. We also assume large number of observations with respect to $d$, that is $n \gg d$. We take the number of components $k$ as known. This puts no constrains on our considerations as the procedure may easily be repeated for all $k$ within the range of interest. The parameters of the mixture are given by $\mu = (1/k) \sum_{l = 1}^k \mu_l$, $\mu \in \mathbb{R}^d$ and $\mathbf{\Sigma} 	= (1/k) \sum_{l = 1}^k \mathbf{\Sigma}_l+ (1/k)\sum_{l = 1}^k (\mu_l -\mu) (\mu_l - \mu)^T$, $\mathbf{\Sigma} \in \mathbb{R}^{d \times d}$. The latter constitutes the covariance decomposition to its within and between cluster component (see \cite{kmb}). 

We assume that each mixture component corresponds to one cluster. A grouping that divides observations into clusters is called a clustering solution or a clustering structure. Note that heterogeneity of covariance matrices allows for varied clusters' shapes, while equal mixing factors imply balanced cluster sizes.

Let $\mu_X \in \mathbb{R}^d$ and $\mathbf{\Sigma}_X \in \mathbb{R}^{d \times d}$ refer to the empirical estimates of the mixture parameters. We assume the covariance matrix to be of full rank, $\rank(\mathbf{\Sigma}_X) = d$. Let $T_X = n \mathbf{\Sigma}_X$ be the \emph{total scatter matrix} for $X$. We say that data is in \emph{isotropic position} if $\mu_X = \mathbf{0}$ and $T_X = \mathbf{I}$.

For symmetric $C \in \mathbb{R}^{d \times d}$ let $C = A_{C}L_{C}A_{C}^T$  be the \emph{spectral decomposition (eigenproblem solution)} for matrix $C$, where
$ L_{C} = \diag(\lambda^{C}_1, \ldots, \lambda^{C}_d)$, $\lambda^{C}_1 \geq \ldots \geq \lambda^{C}_d)$,
is a matrix of eigenvalues for $C$ in a non-decreasing order and $A_{C} = (a^{C}_1, \ldots, a^{C}_d)$ is a matrix of the corresponding column eigenvectors. Alternatively, when considering the eigenproblem for different data sets, we will use the data set as a subscript or superscript (e.g. $C_X = A_{X}L_{X}A_{X}^T$). By $\mathit{PC(k-1)}$ we denote the \emph{principal component subspace} spanned by the first $k-1$ principal components (i.e. $k-1$ eigenvectors of the matrix $\mathbf{\Sigma}_X$ corresponding to its $k-1$ largest eigenvalues, see more in \cite{kmb}, \cite{htf} or \cite{Lipovetsky:2009fk} and references therein for possible extensions).  

By $S^*$ we denote the \emph{Fisher's discriminant (Fisher's subspace)}, which is a $(k-1)$-dimensional subspace that best discriminates $k$ given classes as
\begin{equation*}
	S^* = \operatorname*{argmax}_{\substack{S \subset \mathbb{R}^{d} \\ \dim(S) = k-1}} \frac{\sum\limits^{k-1}_{j = 1} v_j^T B_X v_j }{\sum\limits^{k-1}_{j = 1} v_j^T T_X v_j},
\end{equation*}
where $B_X = \sum_{l = 1}^{k} n_l \left(\mu_{X,l} -\mu_X\right) \left(\mu_{X,l} -\mu_X\right)^T$
is the between cluster component of the total scatter matrix for $X$ with $\mu_{X,l}$ denoting the empirical mean of $l$-th cluster, $l=1, \dots, k$ and $v_1, \ldots, v_{k-1}$ is the orthonormal basis for $S$. Details of this specific definition are given in \cite{fuku}, while the general concept is discussed in \cite{kmb}. 

It is well known that $S^*$ is the subspace spanned by $k-1$ eigenvectors corresponding to the non-zero eigenvalues of a generalized eigenproblem defined by $B_X$ and $T_X$ matrices
\begin{equation}\label{eq:1}
	B_X v = \lambda T_X v,
\end{equation}
which reduces to a standard eigenproblem $T_X^{-1} B_X v = \lambda v$.
Note that the solution is scale-invariant and the eigenvalues are in $[0,1]$ interval. For later reference we note that substituting $\tilde{B} = \left(L_{T_X}^{-1/2} A_{T_X}^{T}\right) B_X \left(L_{T_X}^{-1/2} A_{T_X}^{T}\right)^{T}$ and $\tilde{v} = L_{T_X}^{1/2} A_{T_X}^{T} v$,
we get an equivalent standard eigenproblem for $\tilde{B}$
\begin{equation}\label{eq:1a}
	\tilde{B} \tilde{v} = \lambda \tilde{v}.
\end{equation}

In terms of Fisher's discriminant we define \emph{structure distinctness coefficient} as
\begin{equation}\label{eq:2}
	\bar{\lambda^X} = \frac{1}{k-1} \sum_{j = 1}^{k-1} \lambda_{j}^{T_X^{-1}B_X},
\end{equation}
which is the average eigenvalue over $k-1$ largest eigenvalues of the $T_X^{-1}B_X$ eigenproblem and the mean variability in the Fisher's subspace at the same time. The choice of this particular measure is further explained in Section \ref{dist}.

For all the notation, when it is clear from the context, subscripts and superscripts are omitted.

%%%%%%%% concept
\subsection{Concept}

In principle, the most desirable way to reduce dimension and preserve structure is to project data to $S^*$ which by definition discriminates groups best. However, $S^*$ is defined by cluster structure, so the projection is infeasible if the classes are unknown. On the other hand, a simple projection to $PC(k-1)$ --- which does not require cluster assignments --- may blur the structure as it is shown in the first chart of Fig. \ref{rys:1}. Therefore, the idea is to derive a prior data transformation that makes $PC(k-1)$ approximate $S^*$ and preserves distinctness of the original structure at the same time. PCA on the transformed data is expected to capture the structure well and it is feasible even for unknown classes. As such, it facilitates further structure exploration in the subspace of reduced dimension.

\begin{figure}[!ht]
	\begin{center}
		\includegraphics[width=1\linewidth]{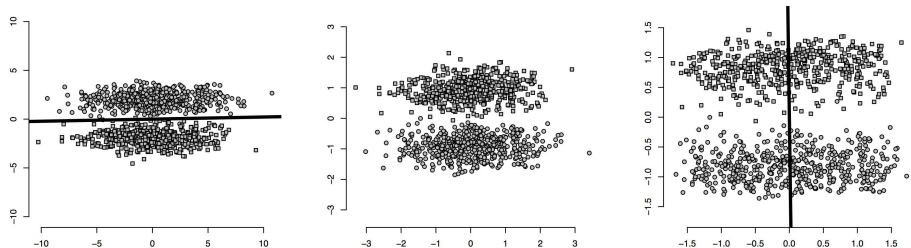}
		\caption{Original data, isotropic data and weighted data respectively, for $k = 2$ and $d = 2$. Principal direction $PC(k-1) = PC(1)$ marked with the black line}
		\label{rys:1}
	\end{center}
\end{figure}

The actual data transformation is divided into two steps referred to as isotropization and weighting. The motivation behind the first one is to bring the mixture to a generic and uniform position that allows for comparisons. Subsection \ref{isopres} shows that this step does not affect distinctness of the structure in data. It can also be noted that for data in isotropic position the Fisher's subspace equals the intermean subspace, which sets an intuitive link between the abstract concept of Fisher's subspace and the tangible notion of cluster centers. However, principal component analysis does not operate on the data of uniform variability (no unique solution). Therefore, the second step is designed to introduce small
perturbation. Namely, it is meant to make the principal components coincide with the directions of best class discrimination and consequently bring $PC(k-1)$ close to $S^*$. At the same time the initial structure distinctness is preserved with only negligible error as it is shown in Subsection \ref{weiperturb}. This concept is illustrated by the last chart of Fig. \ref{rys:1}. Although projection to $PC(k-1) = PC(1)$ carried no information on the clustering structure for the original data, for the transformed data principal direction coincides with the direction of best between cluster discrimination.

Let us emphasize here, that we assume clusters (classes) to be known, which is inevitable to examine the method's properties. However, the ultimate algorithm, of course, operates on raw data only and does not require the knowledge of cluster belongings. Note also, that when speaking of motivation we use theoretical concepts at population level, however the actual calculations are made for given data, i.e. at sample level.

%%%%%%%% model
\subsection{Content}
Section \ref{trans} gives details of the data transformation. It recalls explicit formula for isotropization and justifies the derivation of weights. Section \ref{dist} discusses the characteristics of the structure distinctness coefficient and explains the choice. It also proves that the data transformation affects the structure distinctness only to a negligible extent. Section \ref{diss} focuses on the performance of the method, studying its effect on similarity between the $PC(k-1)$ and $S^*$. Finally, Section \ref{conc} summarizes the findings and points to potential applications of the method.

%%%%%%%%%%%%%%%%%%%%%%% DATA TRANSFORMATION %%%%%%%%%%%%%%%%%%%%%%%
\section{Data transformation}\label{trans}

%%%%%%%% Isotropic transformation
\subsection{Isotropic transformation}

The aim of this step is to transform the data from $X$ to $Y$ so its grand mean is equal to zero (centered) and its scatter matrix is equal to identity matrix (decorrelated). The first step reduces to a simple subtraction of the grand mean
\begin{equation*}
	X_0 =  \left( x_1 - \bar{\mu}^X, \ldots, x_n - \bar{\mu}^X\right)^T,
\end{equation*}
while the second is obtained with help of spectral decomposition of $T_{X_0} =  X_0^T X_0 = A_{T_{X_0}} L_{T_{X_0}} A_{T_{X_0}}^{T} = \left(A_{T_{X_0}} L_{T_{X_0}} ^{\frac{1}{2}}\right) \left(L_{T_{X_0}}^{\frac{1}{2}} A_{T_{X_0}}^{T}\right)$.
Observing that $A_{T_{X_0}}$ is orthonormal and $L_{T_{X_0}}$ is diagonal, we get $\left(X_0 A_{T_{X_0}} L_{T_{X_0}}^{-\frac{1}{2}}\right)^{T} \left(X_0 A_{T_{X_0}} L_{T_{X_0}}^{-\frac{1}{2}}\right) = \mathbf{I}$,
which proves that
\begin{equation}\label{eq:3}
	Y = X_0 A_{T_{X_0}} L_{T_{X_0}}^{-\frac{1}{2}}
\end{equation}
is the required isotropic transformation of $X$.

%%%%%%%% Weighting
\subsection{Weighting}

The second step of data transformation --- from $Y$ to $Z$ --- is required to differentiate variability and make PCA operational. Namely, it is meant to reduce variance in all the directions but the ones that are determined by the cluster centers. As such, it will make the directions of largest overall variability coincide with the directions of best cluster discrimination and consequently bring $PC(k-1)$ close to $S^*$.

The transformation can only distort the clustering structure to a little extent, otherwise it would hamper the inference on the initial structure distinctness level based on the results for the transformed data. The idea, then, is to relocate the extreme observations only, leaving the core of the structure almost untouched. The extreme observations contribute to the total scatter, but they are only of secondary meaning to the general distinctness of the clustering structure.

In order to motivate our choice of the weighting function, let $\omega = (\omega_1, \ldots, \omega_n)$, $\omega \in \mathbb{R}^n$ denote a vector of weights, then $Z = \diag (\omega) Y$. By
\begin{equation}\label{eq:100}
		F_{i,j} = 	\begin{cases}
							1 - \frac{1}{n} & \text{for } i = j, \\
							-\frac{1}{n} & \text{for } i \neq j,
						\end{cases}
\end{equation}
for $F \in \mathbb{R}^{n \times n}$ we define a centering operator (i.e. $Z_0 = FZ$). For a matrix of cluster belongings $E \in \mathbb{R}^{n \times k}$
\begin{equation*}
		E_{i,l} =	\begin{cases}
							1  & \text{for } c(i) = l, \\
							0 & \text{for } c(i) \neq l.
						\end{cases}
\end{equation*}
we define a hat matrix as $H = E (E^TE)^{-1} E^T$. Using this notation we formulate two remarks and the following lemma.

\begin{remark}\label{lem:1}
	For centering operator $F$, the equalities $F^T = F$ and $FF = F$ hold.
\end{remark}

\begin{proof}
	The first equality is due to the matrix symmetry clear from \eqref{eq:100}. The second is based on a simple observation that centering data more than once has no additional effect on it. Alternatively, it may also be proved by a simple calculation using \eqref{eq:100}.
\end{proof}

\begin{remark}\label{rem:1}
	Hat matrix $H = E (E^TE)^{-1} E^T$ is symmetric, semi positive definite and has $k$ non-zero eigenvalues equal to $1$.
\end{remark}

\begin{proof}
	Matrix $H$ is symmetric because
	\begin{equation*}
		H^T = \left(E \left(E^TE\right)^{-1} E^T\right)^T = E \left(E^TE\right)^{-1} E^T=H.
	\end{equation*}
	Using the fact that the eigenvalues for $A \cdot B$ and $B \cdot A$ coincide up to the possible zero eigenvalues we get that the non-zero eigenvalues for $H = E (E^TE)^{-1} E^T$ are equal to the non-zero eigenvalues of $E^T E (E^TE)^{-1} = \mathbf{I}_k$, where $\mathbf{I}_k$ is a $k \times k$ identity matrix. Therefore, $H$ has $k$ non-zero eigenvalues equal to $1$, which means in particular that it is a semi positive definite matrix.
\end{proof}

\begin{lemma}\label{lem:2}
		Total scatter matrix $T_{Z_0}$ and between cluster scatter matrix $B_{Z_0}$ for transformed and centered data $Z_0$ can be expressed in terms of data $Y$ in isotropic position as
	\begin{equation}\label{eq:4a}
		T_{Z_0} = Y^T \diag(\omega) F \diag(\omega) Y
	\end{equation}
and
	\begin{equation}\label{eq:4b}
		B_{Z_0} = Y^T \diag(\omega) H \diag(\omega) Y.
	\end{equation}
\end{lemma}

\begin{proof}
	The proof is purely technical and uses the properties of $F$ and $H$ matrices in the context of the assumed model.
	
	Transformed and centered data $Z_0$ can be expressed as
		\begin{equation}\label{eq:117}
			Z_0 = F \diag{(\omega)} Y
		\end{equation}
		and its total scatter matrix $T_{Z_0}$ --- using Remark \ref{lem:1} --- equals
		\begin{equation}\label{eq:118}
			T_{Z_0} =  Z_0^T Z_0 = \left(F\diag(\omega) Y\right)^T (F \diag(\omega) Y) =  Y^T \diag(\omega) F \diag(\omega) Y,
		\end{equation}
		which proves formula \eqref{eq:4a}. 
	
	Between cluster scatter matrix --- in its corresponding matrix form --- is given by
		\begin{equation}\label{eq:126}
			B_{Z_0} = n M_{Z_0} \diag(\pi) M^T_{Z_0}
		\end{equation}
	for $M_{Z_0}$ a matrix of column vectors of means for subsequent clusters $M_{Z_0} = Z^T_0 E (E^TE)^{-1}$, $M_{Z_0} \in \mathbb{R}^{d \times k}$. Expanding $M_{Z_0}$ in \eqref{eq:126} and using \eqref{eq:117} for expressing $Z_0$ in terms of $Y$ we get
		\begin{multline*}
			B_{Z_0} = n \left(Z_0^T E \left(E^TE\right)^{-1}\right) \diag(\pi) \left(Z_0^T E \left(E^TE\right)^{-1}\right)^T = \\ 
			= n Y^T \diag(\omega) F E \left(E^TE\right)^{-1} \diag(\pi) \left(E^TE\right)^{-1} E^T F \diag(\omega) Y.
		\end{multline*}
	The following equality for balanced cluster sizes
		\begin{equation*}
			n \diag (\pi) \left(E^TE\right)^{-1} = n \diag \left(\frac{1}{k}\right) \left(\diag\left(\frac{n}{k}\right)\right)^{-1}= \diag \left(\frac{n}{k}\right) \left(\diag\left(\frac{n}{k}\right)\right)^{-1} = \mathbf{I}
		\end{equation*}
	reduces the above expression to
		\begin{equation}\label{eq:127}
			B_{Z_0} = Y^T \diag(\omega) F E \left(E^TE\right)^{-1} E^T F \diag(\omega) Y = Y^T \diag(\omega) J \diag(\omega) Y,
		\end{equation}
	for a centered cluster belonging operator $J =  F E (E^TE)^{-1} E^T F$.
		
	The formula \eqref{eq:127} can be further simplified due to the specific properties of the problem considered. A simple calculation shows that if one variable is centered, centering the other one has no impact on their correlation. The same applies for canonical correlation as it is entirely correlation-based. As the generalized eigenproblem defined by matrices $B_Y$ and $T_Y$ (or $B_{Z_0}$ and $T_{Z_0}$ analogously) can be equivalently stated in terms of a CCA problem it can be interpreted as canonical correlation between $Y$ ($Z_0$ alternatively) and cluster belonging matrix denoted by $E$. As we transform the data to be centered, we can assume that $E$ is centered as well, without any impact on the ultimate result of the analysis. As such $FE = E$. It reduces formula \eqref{eq:127} to
		\begin{equation}\label{eq:132}
			B_{Z_0} = Y^T \diag(\omega) E \left(E^TE\right)^{-1} E^T \diag(\omega) Y = Y^T \diag(\omega) H \diag(\omega) Y,
		\end{equation}
	which gives \eqref{eq:4b} and concludes the proof.	
\end{proof}

We proceed now with a series of approximations and transformations that motivate the derivation of the weights. As we intend to introduce only little distortion, we may assume that after the weighting data remains centered approximately at $0$, due to balanced cluster sizes. Thus, the total scatter matrix may be approximated with
\begin{equation*}
	T_{Z_0} = Y^T \diag(\omega) F \diag(\omega) Y \approx Y^T \diag(\omega)^2  Y,
\end{equation*}
as the centering factor can be skipped. To relocate the most distant observations, we draw them closer to the data center, at the rate inversely proportional to their original distance. Note, that for zero-centered data this idea corresponds to equalizing their contribution to the total scatter. The scatter matrix for $Y$ was equal to identity so --- unless significantly distorted by the weighting --- the largest and most meaningful entries remain on the diagonal and the off-diagonal elements exert only negligible effect on the total scatter. The diagonal elements of $T_{Z_0}$ are equal to
\begin{equation*}
	t_{Z_0 j,j} = \sum_{i = 1}^n \omega_i^2 y_{i,j}^2,
\end{equation*}
so their sum over the diagonal --- that corresponds to the total scatter --- equals
\begin{equation*}
	c = \sum_{j = 1}^d t_{Z_0 j,j} = \sum_{j = 1}^d \sum_{i = 1}^n \omega_i^2 y_{i,j}^2 = \sum_{i = 1}^n \omega_i^2 \sum_{j = 1}^d y_{i,j}^2 =  \sum_{i = 1}^n \omega_i^2 \norm{y_i}^2,
\end{equation*}
where $c$ captures the total sum of the elements on the diagonal of the scatter matrix $T_{Z_0}$ and $\norm{\cdot}$ refers to the vector's euclidean norm. Dividing both sides by the constant $c$ we get
\begin{equation*}
	1 =  \sum_{i = 1}^n \omega_i^2 \left(\frac{1}{c} \norm{y_i}^2\right).
\end{equation*}
To maintain the above equality and equalize the contribution of all the observations to the total scatter we take
\begin{equation*}
	\omega_i^2 = \frac{1}{\frac{1}{c} ||y_i||^2}
\end{equation*}
and we modify it adding $1$ in the denominator. On one hand it prevents explosions for small norms, while on the other it guarantees virtually no changes to the very core of the data structure, leaving the central observations untouched
\begin{equation}\label{eq:5}
	\omega_i = \sqrt{\frac{1}{1 + \frac{1}{c} \norm{y_i}^2}} = \sqrt{\frac{1}{1 + \frac{1}{\alpha} \norm{y_i}^2}}.
\end{equation}
As a rule of thumb, the weighting parameter $\alpha$ was fixed at  $\alpha=0.5$, independent from dimensionality $d$, number of clusters $k$ and other data parameters to allow for cross comparisons. It ensures meaningful contribution of observations' individual location, while still granting negligible distinctness' perturbations due to \eqref{eq:12} considered later.

Note, that \cite{vemp} suggests exponential choice of the weighting function
given by
\begin{equation*}
	\omega^{BV}_i = \exp\left(-\frac{\norm{y_i}^2}{2 \beta}\right) =\exp\left(-\frac{\norm{y_i}^2}{\alpha}\right),
\end{equation*}
where $\beta \leq k\cdot d $. For comparison ease let us replace $2\beta = \alpha$. Taylor's expansions for both weighting functions show that their behavior around zero is similar, however for the hyperbolic weighting \eqref{eq:5} the decrease is slightly slower so a larger area of central observations remains untouched. At the same time, for peripheral observations, the values of exponential weighting drop more rapidly with the increase in the observation's original distance. As such, there is less variability in transition values for most distant observations, which leads to more squeezed and spherical data structure. To sum up, for hyperbolic weighting \eqref{eq:5} smaller changes to the central area tend to preserve structure distinctness better, while higher variability in peripheral behavior makes principal components recognize the directions of best cluster discrimination more accurately.

%%%%%%%%%%%%%%%%%%%%%%% STRUCTURE DISTINCTNESS %%%%%%%%%%%%%%%%%%%%%%%
\section{Structure distinctness}\label{dist}

%%%%%%%%%% structure distinctness coefficient
\subsection{Structure distinctness coefficient}

For mixture models, most intrinsic and intuitive structure distinctness coefficient is defined as
\begin{equation}\label{eq:9}
	\sdist = 1 -\int_{\mathbb{R}^d} \min \left(  \pi_1 f_1 \left(\mu_1, \mathbf{\Sigma}_1\right), \pi_2 f_2 \left(\mu_2,\mathbf{\Sigma}_2\right)\right) (x) \text{dx} = 1-\text{MLE}_{\text{err}}, %\,\ud x.
\end{equation}
where $\text{MLE}_{\text{err}}$ stands for probability of misclassification with maximum likelihood estimate (MLE), which equals the integral that captures the area of overlap between the components, $\sdist \in [0,1]$ (for reference see \cite{ab}, \cite{ray}, \cite{sun}). The interpretation and behavior of $\sdist$ is entirely intuitive, however the coefficient is virtually intractable for mixtures of varied covariance (heterogeneous) or higher dimension. Its best linear approximation does not have a closed analytical form either (see \cite{ab}). Therefore, $\sdist$ may only serve as a reference measure and should be replaced with another coefficient that reflects its behavior but is easier to handle analytically. For this purpose, we introduce \eqref{eq:2}, expressed in terms of Fisher's eigenvalues. It captures average variability in Fisher's subspace. As desired, it may only grow with increase in between cluster scatter or decrease in within cluster scatter, as Fisher's task is scale-invariant, and remains within $[0,1]$ interval. Analysis of the relation between the two coefficients showing their strong correspondence can be found in \cite{ja_overlap}. An example directly supporting the choice is presented in Fig. \ref{rys:2}.
\begin{figure}
\begin{center}
\includegraphics[width=0.7\textwidth]{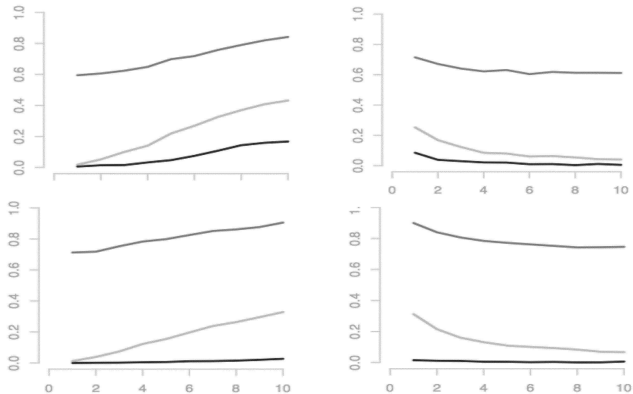}
		\caption{Coefficients of structure distinctness --- integral measure \eqref{eq:9} (top line),  Fisher's average eigenvalue \eqref{eq:2} (middle line), Fisher's minimum non-zero eigenvalue (bottom line) --- effect of increasing between-cluster distance (left panel) and effect of increasing within-cluster dispersion (right panel) }
		\label{rys:2}
\end{center}
\end{figure}

%%%%%%%%%% effect of isotropic transformaiton
\subsection{Preservation under isotropic transformation}\label{isopres}
In a general setup, eigenproblem solution is not preserved under linear transformations. Indeed, for original data $X$ and its linearly transformed counterpart in isotropic position $Y$, the eigenvectors may differ. However, the eigenvalues remain the same.

\begin{lemma}\label{lem:10}
	Isotropic transformation does not change eigenvalues for the Fisher's eigenproblem.
\end{lemma}

\begin{proof}
Consider data $Y$ in isotropic position obtained from the centered data $X_0$ with \eqref{eq:3}. Then, the isotropic transformation for a column vector is given by
	\begin{equation*}
		y = L_{T_X}^{-1/2} A_{T_X}^T x.
	\end{equation*}
As such, matrix $B_X$ becomes $B_{Y} = \left(L_{T_X}^{-1/2} A_{T_X}^{T}\right) B_X \left(L_{T_X}^{-1/2} A_{T_X}^{T}\right)^{T}$ and $T_X$ by definition of isotropic transformation changes into $T_{Y} = \mathbf{I}$. Hence, for data in isotropic position, the generalized eigenproblem \eqref{eq:1} automatically reduces to a standard eigenproblem
	\begin{equation*}
		B_{Y} y = \lambda y.
	\end{equation*}
	As $B_Y = \tilde{B} $, the above equation corresponds to \eqref{eq:1a} and yields the same solution, in particular the eigenvalues are the same for both problems.

\end{proof}

As structure distinctness is defined by \eqref{eq:2} as an average eigenvalue for the Fisher's eigenproblem, the following corollary holds.
\begin{corollary}\label{cor:1}
Isotropic transformation does not affect structure distinctness defined by \eqref{eq:2}.
\end{corollary}

%%%%%%%%%% effect of weighting
\subsection{Effect of weighting}\label{weiperturb}

In this subsection we show that the effect of weighting on the structure distinctness can only be negligible. 

We start with a technical Lemma \ref{lem:20}, which shows that squared norms of observations $y_i$ are small on average.

\begin{lemma}\label{lem:20}
	For data $Y= (y_{i,j})_{\substack{i = 1, \ldots, n \\ j = 1, \ldots, d}} $ in isotropic position we have
	\begin{equation}\label{eq:6a}
		\overline{\norm{y_i}^2} = \frac{d}{n} \ll 1
	\end{equation}
\end{lemma}

\begin{proof}
	For data in isotropic position $\sum_{i=1}^n \norm{y_i}^2 = \sum_{j = 1}^{d} \left( \sum_{i = 1}^{n} y_{i,j}^2 \right) = d$.
\end{proof}

Note, that the average value of \eqref{eq:6a} is very small. It is difficult to prove analytically, but the simulations show that its standard deviation is very small with respect to the mean value \eqref{eq:6a} either. Hence, we believe it is justified to assume that the standard deviation at least shares the upper bound with the mean value. Accordingly, we assume in the sequel, that $o (\norm{y_i}^2)$ is negligible and the standard deviation of $\norm{y_i^2}$ satisfies $\sd (\norm{y_i^2}) < \frac{d}{n}$. In view of this, Taylor's expansion provides the following linear approximation of the weighting function for $\Delta = \frac{1}{2 \alpha} \diag (\norm{y_i}^2)$
\begin{equation}\label{eq:7}
		\diag(\omega) = \mathbf{I} - \frac{1}{2 \alpha} \diag \left(\norm{y_i}^2\right) + o\left(\diag \left(\norm{y_i}^2\right)\right) \approx \mathbf{I} - \Delta.
\end{equation}

Next, we show that the total and between scatter matrices for weighted data $Z_0$ can be represented as slightly perturbed corresponding matrices for isotropic data $Y$.

\begin{lemma}\label{lem:3}
	For $T_{Z_0}$ given by \eqref{eq:4a} and $B_{Z_0}$ given by \eqref{eq:4b} we have
	\begin{equation}\label{eq:8a}
		T_{Z_0} = T_Y + \delta T_Y \mbox{ where } \delta T_Y \ll T_Y
	\end{equation}
and
	\begin{equation}\label{eq:8b}
		B_{Z_0} = B_Y + \delta B_Y \mbox{ where }  \delta B_Y \ll B_Y.
	\end{equation}
\end{lemma}

\begin{proof}
	The proof is direct and uses linear approximation of weights to facilitate matrices' manipulation.
	
	For $T_{Z_0}$ given by \eqref{eq:4a} linear approximation of weights yields directly
		\begin{multline}\label{eq:133}
			T_{Z_0} = Y^T \text{diag}(\omega) F \text{diag}(\omega) Y \approx Y^T (\mathbf{I} - \Delta) F (\mathbf{I} - \Delta) Y  = \\
			= Y^T F Y - Y^T F\Delta Y - Y^T \Delta F Y + o(\Delta) \approx Y^T F Y - Y^T F\Delta Y - Y^T \Delta F Y = \\
			=  Y^T Y - Y^T \Delta Y - Y^T \Delta Y = T_Y - 2Y^T \Delta Y = T_Y + \delta T_Y,
		\end{multline}
	as $Y$ is already centered, $FY = Y$. Due to smallness of perturbation $\Delta$, the  quadratic form can be omitted and $\delta T_Y$ can be considered small indeed. The same holds for between cluster scatter matrix, so analogously for matrix $B_{Z_0}$ given by \eqref{eq:4b} we get
		\begin{multline}\label{eq:130}
			B_{Z_0} = Y^T \text{diag}(\omega) H \text{diag}(\omega) Y \approx Y^T (\mathbf{I} - \Delta ) H (\mathbf{I} - \Delta ) Y =\\
			=  Y^T H Y - Y^T  H \Delta Y - Y^T  \Delta  H Y + o(\Delta)  \approx Y^T H Y - Y^T  H \Delta Y - Y^T  \Delta H Y = \\
			= B_Y - Y^T  H \Delta Y - Y^T  \Delta H Y = B_Y + \delta B_Y,
		\end{multline}
	which concludes the proof.
\end{proof}

For slightly perturbed eigenproblem as in Lemma \ref{lem:3}, the following lemma gives explicit formulas for eigenvalues and their corresponding eigenvectors in terms of the solution for the original eigenproblem.

\begin{lemma}[Eigenproblem perturbation]\label{lem:4}
	For symmetric and semi positive definite matrices $K_0, M_0 \in \mathbb{R}^{d \times d}$ we consider a generalized eigenproblem
	\begin{equation*}
		K_0 a^0_j = \lambda^0_j M_0 a^0_j \quad \text{for} \quad j = 1, \ldots d
	\end{equation*}
	and its perturbation
	\begin{equation*}
		K a_j = \lambda_j M a_j \quad \text{for} \quad j = 1, \ldots d,
	\end{equation*}
	with $K = K_0 + \delta K_0$ and $M = M_0 + \delta M_0$, where the perturbation is assumed to be small $\delta K_0 \ll K_0$ and $\delta M_0 \ll M_0$. Then the eigenvalues $\lambda_j$ and eigenvectors $a_j$ of the perturbed problem can be expressed in terms of the original eigenvalues $\lambda^0_1, \ldots \lambda^0_d$ and eigenvectors $a^0_1, \ldots \lambda^0_d$ as follows
	\begin{multline}\label{eq:10a}
		\lambda_j = \lambda^0_j + \delta \lambda^0_j = \lambda^0_j + \delta \lambda^{0 (I)}_j  + o(\delta \lambda^0_j) \approx \\
		\approx \lambda^0_j + \delta \lambda^{0 (I)}_j = \lambda^0_j + \left(a^0_j\right)^T \left(\delta K_0 - \lambda^0_j (\delta M_0)\right) \left(a^0_j\right)
	\end{multline}
	and
	\begin{multline}\label{eq:10b}
		a_j = a^0_j + \delta a^0_j = a^0_j + \delta a^{0 (I)}_j + o(\delta a^0_j) \approx a^0_j + \delta a^{0 (I)}_j = \\
		=  a^0_j\left(1-\frac{1}{2}\left(a^0_j\right)^T \left(\delta M_0\right) \left(a^0_j\right)\right) + \sum_{\substack{i = 1 \\ i \neq j}}^{d} \frac{\left(a^0_i\right)^T \left(\delta K_0 - \lambda^0_i \delta M_0\right) \left(a^0_i\right)}{\lambda^0_j - \lambda^0_i},
	\end{multline}
	where the superscript $(I)$ denotes first order term. Higher order terms are omitted as negligible due to the assumption of small perturbation.
\end{lemma}

\begin{proof}
	Proof can be found for instance in \cite{perturb}.
\end{proof}

\begin{corollary}\label{cor:3}
Eigenvalues and eigenvectors for generalized eigenproblem with matrices $B_{Z_0}$ and $T_{Z_0}$ (Fisher's task) can be expressed in terms of perturbed eigenvalues and eigenvectors of the problem given by $B_Y$ and $T_Y$  following the formulas of Lemma \ref{lem:4}.
\end{corollary}

Now, let us recall several facts on matrix norms that will be used in the course of the proposition's proof.
\begin{remark}\label{lem:9}
	For a symmetric matrix $A = (a_{i,j})_{\substack{i = 1, \ldots, d \\ j = 1, \ldots, d}}$, $A \in \mathbb{R}^{d \times d}$, let $\abs{\lambda^A_{\max}}$ denote the maximum absolute value of the eigenvalues of $A$. Let
	\begin{enumerate}
		\item $\norm{A}_S = \abs{\lambda^A_{\max}}$ define and denote \textbf{spectral norm} of matrix $A$,
		\item $\norm{A}_F = \sqrt{\sum_{i = 1}^d \sum_{j = 1}^d a_{ij}^2} = \sqrt{\tr(A A^T)}$ define and denote \textbf{Frobenius norm} of matrix $A$.
	\end{enumerate}
	Then $\norm{A}_S \leq \norm{A}_F$ and for any vector $x \in \mathbb{R}^d$ we have $\abs{x^T A x} \leq \abs{\lambda^A_{\max}} \norm{x}$, which yield together
	\begin{equation}\label{eq:156}
		\abs{x^T A x} \leq \abs{\lambda^A_{\max}} \norm{x} = \norm{A}_S \norm{x} \leq \norm{A}_F \norm{x}.
	\end{equation}
\end{remark}

\begin{proof}
	Proof can be found for instance in \cite{rayleigh}.
\end{proof}

\begin{remark}\label{rem:2}
	For two symmetric matrices $A,B \in \mathbb{R}^{d \times d}$ and a constant $c \in \mathbb{R}$ by norm definition the following conditions are fulfilled
	\begin{enumerate}
		\item $\norm{c A} \leq \abs{c} \norm{A}$
		\item $\norm{A+B} \leq \norm{A} + \norm{B}$.
	\end{enumerate}
	Additionally, for Frobenius norm submultiplicative condition is fulfilled (also see \cite{rayleigh})
	\begin{enumerate}
		\item[(c)] $\norm{AB}_F \leq \norm{A}_F \cdot \norm{B}_F$.
	\end{enumerate}
\end{remark}

Now, let us formulate the main proposition that gives the upper bound on the difference between structure distinctness for original and transformed data. Although stated in terms of $X$ and $Z$ data it actually captures the effect of weighting as isotropization does not affect it in any way.

\begin{proposition}\label{thm:1}
	In agreement with our previous notation and assumptions
	\begin{equation}\label{eq:12}
		\abs{\bar{\lambda^Z} - \bar{\lambda^X}} \leq \frac{1}{\sqrt{n}} \left( \frac{d}{ \alpha} \left(\bar{\lambda^X} + \sqrt{k}\right) \right).
	\end{equation}
\end{proposition}

\begin{proof}
	Weighting would not affect structure distinctness if the weights were equal, as the Fisher's task is scale invariant. Therefore, possible perturbation in structure distinctness is entirely due to the variance of weights which can be claimed to be very small (as earlier mentioned, we found it justified to assume that $\norm{y_i^2}$ satisfies $\sd (\norm{y_i^2}) < \frac{d}{n}$). As such, the idea of the proof is to translate the small variance of weights into possible perturbation of the resulting structure distinctness and provide an upper bound on it. For that purpose Corollary \ref{cor:3} is used and a linear approximation of the weights together with basic matrix norm properties lead to the final approximation. %Detailed calculations are given in \cite{report}.
	
	To estimate the difference between $\lambda_j^Z$ and $\lambda_j^Y$ we use perturbation formula \ref{eq:10a}. For generalized Fisher's eigenproblem it takes the form
		\begin{equation}\label{eq:157}
			\lambda^Z_j = \lambda^Y_j + \left(a_j^Y\right)^T \left(\delta B_Y - \lambda^Y_j \delta T_Y\right)\left(a_j^Y\right).
		\end{equation}
	From \eqref{eq:133} and \eqref{eq:130} we have
		\begin{equation*}
			\delta T_Y = -2Y^T \Delta Y 
		\end{equation*}
	and
		\begin{equation*}
			\delta B_Y = -Y^T H \Delta Y - Y^T \Delta H Y,
		\end{equation*}
	so the difference becomes
		\begin{multline*}
			\left(\delta B_Y - \lambda^Y_j \delta T_Y\right) = 2 \lambda^Y_j Y^T \Delta Y - Y^T H \Delta Y - Y^T \Delta H Y = \\
			= Y^T \left( \Delta\left(\lambda^Y_j I - H\right) + \left(\lambda^Y_j I - H\right) \Delta \right) Y.
		\end{multline*}
	From \eqref{eq:157}
		\begin{multline}\label{eq:158}
			\abs{\lambda^Z_j -  \lambda^Y_j}= \abs{\left(a_j^Y\right)^T \left(\delta B_Y - \lambda^Y_j \delta T_Y\right)\left(a_j^Y\right)} =  \\
			= \abs{\left(a_j^Y\right)^T \left(Y^T \left(\Delta\left(\lambda^Y_j I - H\right) + \left(\lambda^Y_j I - H\right) \Delta \right) Y\right)\left(a_j^Y\right)}.
		\end{multline}
	Using the fact that $\Delta = (1/{2 \alpha}) \text{diag} (\norm{y_i}^2)$ we get
		\begin{multline*}
			\abs{\left(a_j^Y\right)^T \left(Y^T \left(\Delta\left(\lambda^Y_j I - H\right) + \left(\lambda^Y_j I - H\right) \Delta \right) Y\right)\left(a_j^Y\right)} = \\
			=  \frac{1}{2 \alpha}\left\lvert\left(a_j^Y\right)^T \left(Y^T \left(\left(\diag\left(\norm{y_i}^2\right)\right)\left(\lambda^Y_j I - H\right)\right.\right.\right. +\\
			+\left.\left.\left. \left(\lambda^Y_j I - H\right) \left(\diag\left(\norm{y_i}^2\right)\right)\right)Y\right)\left(a_j^Y\right)\right\rvert,
		\end{multline*}
	then adding to and subtracting from $\text{diag} (\norm{y_i}^2)$ the same constant $d/n$ we have
		\begin{multline*}
			\frac{1}{2 \alpha}\left\lvert\left(a_j^Y\right)^T \left(Y^T \left(\left(\diag\left(\norm{y_i}^2 - \frac{d}{n} + \frac{d}{n}\right)\right)\left(\lambda^Y_j I - H\right)\right.\right.\right. +\\
			+\left.\left.\left. \left(\lambda^Y_j I - H\right) \left(\diag\left(\norm{y_i}^2 - \frac{d}{n} + \frac{d}{n}\right)\right)\right)Y\right)\left(a_j^Y\right)\right\rvert
		\end{multline*}
	which splits into
		\begin{multline*}
			\frac{1}{2 \alpha}\left\lvert\left(a_j^Y\right)^T \left(Y^T \left(\left(\diag\left(\norm{y_i}^2 - \frac{d}{n}\right)\right)\left(\lambda^Y_j I - H\right)\right.\right.\right. +\\
			+\left.\left. \left(\lambda^Y_j I - H\right) \left(\diag\left(\norm{y_i}^2 - \frac{d}{n} \right)\right)\right)Y\right)\left(a_j^Y\right) + \\
			+ \left.\left(a_j^Y\right)^T \left(Y^T \left(\frac{d}{n} \left(\lambda^Y_j I - H\right) + \left(\lambda^Y_j I - H\right) \frac{d}{n} \right) Y\right)\left(a_j^Y\right)\right\rvert.
		\end{multline*}
	The last term equals zero as $\lambda^Y_j$ is the eigenvalue of $B_Y a_j^Y = \lambda^Y_j T_Y a_j^Y$, which is equivalent to $Y^T H Y a_j^Y = \lambda^Y_j a_j^Y$ due to the definition of $B_Y$ and the fact that for the data in isotropic position $Y^T Y = T_Y= I$. As such, its characteristic polynomial equals zero at $\lambda^Y_j$ so
		\begin{multline*}
			\left(a_j^Y\right)^T \left(Y^T \left(\frac{d}{n}\left(\lambda^Y_j I - H\right) + \left(\lambda^Y_j I - H\right) \frac{d}{n} \right)Y\right)\left(a_j^Y\right) = \\
			= \frac{d}{n}\left(a_j^Y\right)^T \left(Y^T \left(\left(\lambda^Y_j I - H\right) + \left(\lambda^Y_j I - H\right) \right)Y\right)\left(a_j^Y\right)  = \\
			= \frac{d}{n}\left(a_j^Y\right)^T \left(\left(\lambda^Y_j Y^T Y - Y^T H Y\right) + \left(\lambda^Y_j Y^T Y - Y^T H Y\right)\right) \left(a_j^Y\right) = \\
			= \frac{d}{n}\left(a_j^Y\right)^T \left(\left(\lambda^Y_j I - B_Y\right) + \left(\lambda^Y_j I - B_Y\right)\right) \left(a_j^Y\right) =0.
		\end{multline*}
	It remains to give the upper bound on the first term. As $Y$ is in isotropic position and $a_j^Y$ is standardized as an eigenvector, we have
		\begin{equation*}
			\norm{Y a_j^Y} = \sqrt{\left(Y a_j^Y\right)^T \left(Y a_j^Y\right)} = \sqrt{\left(a_j^Y\right)^T Y^T Y a_j^Y} = \sqrt{\left(a_j^Y\right)^T a_j^Y} = 1.
		\end{equation*}
	Then, using formula \eqref{eq:156} from Remark \ref{lem:9} for 
		\begin{equation*}	
			A = \left(\left(\diag\left(\norm{y_i}^2 - \frac{d}{n}\right)\right) \left(\lambda^Y_j I - H\right) + \left(\lambda^Y_j I - H\right) \left(\diag\left(\norm{y_i}^2 - \frac{d}{n}\right)\right)\right)
		\end{equation*}
	and $x = Y a_j^Y$ we obtain
		\begin{multline*}
			\frac{1}{2 \alpha}\left\lvert\left(a_j^Y\right)^T \left(Y^T \left(\left(\diag\left(\norm{y_i}^2 - \frac{d}{n}\right)\right)\left(\lambda^Y_j I - H\right)\right.\right.\right. +\\
			+\left.\left.\left. \left(\lambda^Y_j I - H\right) \left(\diag\left(\norm{y_i}^2 - \frac{d}{n} \right)\right)\right)Y\right)\left(a_j^Y\right)\right\rvert \leq \\
			\leq \frac{1}{2 \alpha} \norm{\left(\diag\left(\norm{y_i}^2 - \frac{d}{n}\right)\right) \left(\lambda^Y_j I - H\right) + \left(\lambda^Y_j I - H\right) \left(\diag\left(\norm{y_i}^2 - \frac{d}{n}\right)\right)}_F.
		\end{multline*}
	Next, rearranging the elements and using additive (b) and submultiplicative (c) norm properties from Remark \ref{rem:2} we get
		\begin{multline}\label{eq:59}
			\frac{1}{2 \alpha} \norm{\left(\diag\left(\norm{y_i}^2 - \frac{d}{n}\right)\right) \left(\lambda^Y_j I - H\right) + \left(\lambda^Y_j I - H\right) \left(\diag\left(\norm{y_i}^2 - \frac{d}{n}\right)\right)}_F = \\ 
			=\frac{1}{2 \alpha} \left\lVert\lambda^Y_j \left(\diag\left(\norm{y_i}^2 - \frac{d}{n}\right)\right) - \left(\diag\left(\norm{y_i}^2 - \frac{d}{n}\right)\right)H\right. + \\
			+ \left.\lambda^Y_j \left(\diag\left(\norm{y_i}^2 - \frac{d}{n}\right)\right) - H\left(\diag\left(\norm{y_i}^2 - \frac{d}{n}\right)\right)\right\rVert_F  \leq \\
			\leq \frac{\lambda^Y_j}{\alpha} \norm{\diag\left(\norm{y_i}^2 - \frac{d}{n}\right)}_F + \frac{1}{\alpha} \norm{\diag\left(\norm{y_i}^2 - \frac{d}{n}\right)}_F \norm{H}_F.
		\end{multline}
	Due to the formula (b) from Remark \ref{lem:9} for Frobenius norm and hat matrix properties we have
		\begin{equation*}
			\norm{H}_F = \sqrt{\tr\left(HH^T\right)} = \sqrt{\tr\left(H^2\right)} = \sqrt{k}.
		\end{equation*}
	We have $\tr(H^2) = \tr(H)$ as a sum of squared eigenvalues of $H$ which are equal $1$ or $0$ in this case. For the other term, from Frobenius norm definition in Remark \ref{lem:9} (b) and the crude estimate for the standard deviation, we get
		\begin{multline*}
			\norm{\diag\left(\norm{y_i}^2 - \frac{d}{n}\right)}_F = \sqrt{\sum_{i = 1}^n \left(\norm{y_i}^2 - \frac{d}{n} \right)^2} = \sqrt{n \cdot \var\left(\norm{y_i}^2\right)} \\
			= \sqrt{n} \cdot \sd \left(\norm{y_i}^2\right) \leq \sqrt{n} \frac{d}{n} = \frac{d}{\sqrt{n}}.
		\end{multline*}
	Now, substituting the above two inequalities into \eqref{eq:59} yields
		\begin{multline*}
			\frac{\lambda^Y_j}{\alpha} \norm{\diag\left(\norm{y_i}^2 - \frac{d}{n}\right)}_F + \frac{1}{\alpha} \norm{\diag\left(\norm{y_i}^2 - \frac{d}{n}\right)}_F \norm{H}_F \leq \\
			\leq \frac{1}{\alpha}  \frac{d}{\sqrt{n}} \left(\lambda^Y_j + \sqrt{k}\right) = \frac{1}{\sqrt{n}} \left(\frac{ d}{\alpha}\left(\lambda^Y_j + \sqrt{k}\right)\right).
		\end{multline*}
	So using all the above estimation for \eqref{eq:158} we get
		\begin{equation*}
			\abs{\lambda^Z_j -  \lambda^Y_j} \leq \frac{1}{\sqrt{n}} \left(\frac{ d}{\alpha}\left(\lambda^Y_j + \sqrt{k}\right)\right) \qquad \text{for } j = 1, \ldots, d.
		\end{equation*}
	After averaging over the $k-1$ non-zero eigenvalues and using the fact that isotropic transformation does not change structure distinctness it yields
		\begin{equation*}
			\abs{\bar{\lambda^Z} -  \bar{\lambda^X}} = \abs{\bar{\lambda^Z} -  \bar{\lambda^Y}} \leq \frac{1}{\sqrt{n}} \left(\frac{d}{\alpha}\left(\bar{\lambda^Y} + \sqrt{k}\right)\right) = \frac{1}{\sqrt{n}} \left(\frac{d}{\alpha}\left(\bar{\lambda^X} + \sqrt{k}\right)\right)
		\end{equation*}
	and concludes the proof.
	
\end{proof}

Since the sample size $n$ is assumed to be very large with respect to the number of dimensions $d$ and the number of clusters $k$, the resulting value of the upper bound in Proposition \ref{thm:1} is very small. It implies that the original clustering structure is affected by the data transformation only to a very little extent and the prior distinctness level is preserved. First, it prevents structure destruction due to the data transformation. Second, it shows that structure distinctness assessments and comparisons made for transformed data sets allow for drawing conclusions for the original data sets. Simulation studies confirm negligible effect of the transformation.

%%%%%%%%%%%%%%%%%%%%%%% SIMILARITY BETWEEN SPACES %%%%%%%%%%%%%%%%%%%%%%%
\section{Similarity between subspaces}\label{diss}

%%%%%%%%%% similarity coefficient
\subsection{Similarity coefficient}

The concept of similarity between spaces is used to assess the difference between $PC(k-1)$ and the reference projection to $S^*$. Projections are not affected by the possible point of origin so we assume linear, not affine, structure only. Without the need for triangle inequality, a similarity measure suffices and a distance is not required.

The problem of subspace similarity assessment is vital for subspace methods gaining popularity in image recognition and face recognition in particular. Works on the topic start with \cite{yamaguchi}, which uses smallest principal angle (see \cite{hot}). Further developments are due to Wolf and Shashua (see \cite{wolf1} and \cite{wolf2}), who utilize sum of squared cosines of principal angles. We make a small variation with respect to \cite{wolf1} and instead of the sum, we utilize the mean to remain within $[0,1]$ interval. It facilitates interpretation and comparisons between different data sets. We use canonical correlations (see \cite{hot} or \cite{kmb}), which are equivalent to squared cosines of principal angles as long as the data is centered. It makes a multi-dimensional generalization of most intuitive squared cosine measure.

To give an explicit formula, we state the canonical correlation task between the two sets of $(k-1)$ column vectors --- matrix $V \in \mathbb{R}^{d \times (k-1)}$ and matrix $A \in \mathbb{R}^{d \times (k-1)}$ that span Fisher's $S^*$ and $PC(k-1)$ subspaces respectively --- in terms of an eigenproblem as $\left( (V^T V)^{-1} (V^T A) (A^T A)^{-1} (A^T V) \right) U =  U L^2$, where $U$ consists of column eigenvectors and $L^2 \in \mathbb{R}^{(k-1) \times (k-1)}$ contains squared canonical correlations on its diagonal or squared cosines of principal angles in other words (for standard Lagrangian derivation, see \cite{kmb}). So we measure subspace similarity (sss) between $V$ and $A$ as
\begin{equation}\label{eq:11}
	\sss(V,A) = \frac{1}{k-1} \sum_{l = 1}^{k-1} L^2(l,l).
\end{equation}
Similarly to simple squared cosine, it takes values from $[0,1]$ interval and increases as similarity does. In other words, the larger the value of \eqref{eq:11}, the more similar the spaces.

%%%%%%%%%% effect of data transformation
\subsection{Effect of data transformation}

The effect of data transformation on the similarity between Fisher's and $PC(k-1)$ subspaces was studied by means of simulation study. The data was generated according to the model assumptions and for each set of data parameters ($d$, $k$ and $n$) the procedure was repeated $50$ times to allow variability for each mixture parameter configuration.

\begin{figure}
\begin{center}
 \includegraphics[width=1\textwidth]{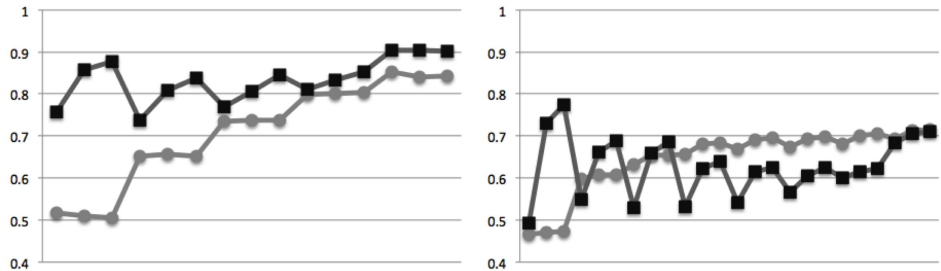}
		\caption{Similarity between spaces \eqref{eq:11} for $X$ (gray) and $Z$ (black) data, for $d=7$ (left chart) and $d=20$ (right chart), increasing triples correspond to $n=100, 300, 500$ per cluster, while subsequent triples to $k=3, \ldots, \min (d,10)$ }
		\label{rys:3}
\end{center}
\end{figure}

It can be observed that the transformation performs best for small number of clusters $k$ in a space of small dimension $d$. As shown in Fig. \ref{rys:3}, after the transformation the subspaces practically overlap. For larger $d$ this is not necessarily the case. There is substantial increase in the value of \eqref{eq:11} for small $k$ and then for large $k$ with respect to $d$ there is almost no change due to little flexibility in dimension reduction. However, in between even substantial drop in average canonical correlation is possible, as it can again be observed in Fig. \ref{rys:3}. What is worth mentioning though, is that the sample size has remarkable impact on the behavior of the average canonical correlation, which is understandable due to sparsity in higher dimensions. The increasing triples in Fig. \ref{rys:3} are all due to increasing sample size - the larger the sample the more significant the increase in average canonical correlation. Therefore, the above mentioned effect of similarity drop can be excluded by taking sample size large enough. It was observed that for $d=20$ sample size of $1500 - 2000$ per cluster prevents correlation drops even for moderate $k$. In other words, for sample size large enough meaningful increase but no significant decrease in average canonical correlation can be observed.

%%%%%%%%%%%%%%%%%%%%%%% CONCLUSIONS %%%%%%%%%%%%%%%%%%%%%%%
\section{Conclusions}\label{conc}

In this work a new method for distinctness preserving dimension reduction is proposed. It is based on a preliminary data transformation that allows Fisher's subspace to be approximated by means of PCA, which does not require the knowledge of data structure or partition. At the same time, the transformation perturbs original distinctness of the classes' structure only to a negligible extent. As such, it facilitates further structure learning in the space of reduced dimension, including assessment of the potential distinctness of the unknown structure. 

The similarity between the two subspaces of interest --- Fisher's $S^*$ requiring data partition and $PC(k-1)$ based on overall variability only --- tend to suffer from increasing space dimension $d$. Depending on the sample size and particular task considered, the acceptable values of $d$ may differ. In particular, if the number of clusters is small, the method is expected to perform well, regardless of the original space dimension. This leaves it with a wide range of possible applications, where space dimension can be preliminarily reduced and/or solutions of few clusters are required.

Although the method already presents a closed tool that may be successfully applied for a certain class of problems, it still needs further investigation that would provide insight in its limitations and possibly support its further development.

% % % % % % % % % % % % % % % % % % % % % % % % % % % % % % % % % % % % % % 
% % % % % % % ------- ACKNOWLEDGEMENTS  --------% % % % % % % % % %
% % % % % % % % % % % % % % % % % % % % % % % % % % % % % % % % % % % % % % 

\section{Acknowledgements}

This work was supported by National Science Center of Poland, grant number DEC-2011/01/N/ST6/04174.

%%%%%%%%%%% --------- BIBLIOGRAPHY  -------- %%%%%%%%%%%%%%%

%\begin{thebibliography}{99}
%
%\bibitem{ja_overlap} E. Nowakowska, J. Koronacki, S. Lipovetsky, \textit{Tractable Measure of Component Overlap \\ for Gaussian Mixture Models}, submitted for publication, available at \verb+http://www.ipipan.waw.pl/~ewano/technical_report.pdf+.
%
%\end{thebibliography}

\bibliographystyle{elsarticle-num}
\bibliography{mybibliography}

\end{document}